\documentclass[12pt]{article}
\usepackage[english]{babel}
\usepackage[utf8]{inputenc}
\usepackage{amstext}
\usepackage{fancyhdr}
\usepackage{amsfonts,graphicx,bezier, amssymb}
\usepackage{amsmath}
\usepackage{caption}
\usepackage{mathtools}
\usepackage{tikz}
\usepackage{rotating}
\usepackage[numbers]{natbib}
\usetikzlibrary{arrows.meta}
\newenvironment{prf}{\noindent{\bf{Proof:}}~~}{\hfill\rule{1ex}{1ex}\vskip1.5ex}

\newcommand{\Z}{\mathbb Z}

\newcommand{\beqa}{\begin{eqnarray}}
\newcommand{\enqa}{\end{eqnarray}}
\newcommand{\beq}{\begin{eqnarray*}}
\newcommand{\enq}{\end{eqnarray*}}

\newcommand{\noi}{\noindent}

\newtheorem{qn}{Question}[section]
\newtheorem{rem}{Remark}[section]

\newtheorem{cor}{Corollary}[section]
\newtheorem{propn}{Proposition}[section]
\newtheorem{defn}{Definition}[section]
\newtheorem{exam}{{\bf Example}}[section]
\newtheorem{thm}{Theorem}[section]

\makeatletter
\providecommand*{\twoheadrightarrowfill@}{%
  \arrowfill@\relbar\relbar\twoheadrightarrow
}
\providecommand*{\twoheadleftarrowfill@}{%
  \arrowfill@\twoheadleftarrow\relbar\relbar
}
\providecommand*{\xtwoheadrightarrow}[2][]{%
  \ext@arrow 0579\twoheadrightarrowfill@{#1}{#2}%
}
\providecommand*{\xtwoheadleftarrow}[2][]{%
  \ext@arrow 5097\twoheadleftarrowfill@{#1}{#2}%
}
\makeatother

\begin{document}
\begin{center}
		{\bf\Large Locally prime modules}

		\end{center}

		 \vspace*{0.2cm}

\begin{center}Sholastica Luambano\footnote{Department of Mathematics, University of Dodoma, P.O. BOX 259, Dodoma, Tanzania} and David Ssevviiri\footnote{Department of Mathematics, Makerere University, P.O. BOX 7062, Kampala, Uganda}$^{,}$\footnote{Corresponding author}  \\
			E-mail:   scholastica.luambano@udom.ac.tz and david.ssevviiri@mak.ac.ug

		\end{center}

	\vspace*{0.1cm}
\begin{abstract} For a commutative unital ring  $R$ with fixed ideals  $I$ and $J$, we introduce and study $I$-prime $R$-modules and  $(I, J)$-prime $R$-modules together with  their duals  $I$-coprime  $R$-modules and  $(I,J)$-coprime $R$-modules respectively. We employ category-theoretic techniques to reveal their  structural properties. Our main results are versions of   the Greenlees-May Duality and the Matlis-Greenlees-May Equivalence to the setting of these prime  and coprime modules. This   generalizes  work on $I$-reduced modules and  $I$-coreduced modules. We  demonstrate that these ``locally prime" modules serve as a  tool for studying the   classical ``globally prime" modules, creating a bridge between local and global primality.
	\end{abstract}


	\vspace*{0.3cm}

{\bf MSC 2020} Mathematics Subject Classification: 13C13, 18A40, 16N60, 13J10

 \section{Introduction}
\begin{paragraph}\noi
 Prime ideals of rings are important in commutative algebra, in algebraic number theory  and  in algebraic geometry. The Krull dimension of a ring which  is the supremum of chains of prime ideals of that ring forms an important invariant. The Krull dimension of a coordinate ring coincides with  the dimension of the associated variety. The collection of all prime ideals  of a ring $R$,  called the {\it spectrum of  $R$} and  denoted by $\text{Spec}(R)$, is a geometric scheme which is in a one-to-one correspondence with the unital commutative ring $R$. Matlis' Theorem states that if $R$ is a commutative Noetherian ring, then there is a bijection between the prime ideals $\mathcal{P}$ of $R$ and the indecomposable injective $R$-modules of the form $E(R/\mathcal{P})$, where $E(R/\mathcal{P})$ denotes the injective hull of the $R$-module $R/\mathcal{P}$.  In number theory, prime ideals are useful  because they generalize prime numbers; enabling the unique factorization of ideals in Dedekind domains where unique factorization of elements fails. This  provides a framework for analyzing and solving problems; such as understanding the structure of number fields and the behavior of primes in extensions.  
 The utility of prime ideals  motivated their generalization to modules defined over rings \cite{BehboodiWP2004,Dauns1978} as well as to other algebraic structures such as near-rings \cite{Groenewald1991} and recently to skew braces \cite{Konovalov2021}. Throughout the paper,  $R$ will be a commutative unital ring and modules  $M$ will be left, unital modules defined over $R$. We denote the category of $R$-modules by $R$-Mod.
\end{paragraph}

 \subsection{Globally  prime modules}\label{1.1}
 \begin{paragraph}\noindent  A module $M$ over a ring $R$ is $\textit{prime}$ if for all ideals $I$ of $R$ and  submodules $N$ of $M$, $IN=0$ implies that either $N=0$ or $IM=0$. A submodule $P$ of  an $R$-module $M$ is a \textit{prime submodule}  if $M/P$ is a prime $R$-module. The study of prime modules was initiated  by Dauns \cite{Dauns1978} in 1978 and became so active in the last four decades, see for instance \cite{AnsariToroghy2011,Azizi2008,Azizi2009, Behbood2004,Behboodi2007, Behboodi2009,BehboodiWP2004,Tiras1999, Yassemi2001} and the references therein. There is a generalization of prime modules called classical prime modules \cite{Behboodi2007} and also called weakly prime modules in the literature, see  \cite{Azizi2008, Azizi2009, BehboodiWP2004}. An $R$-module $M$ is \textit{weakly prime} if for all ideals  $I$ and $J$ of $R$ and  elements $m\in M$, $IJm=0$ implies that either $Im=0$ or $Jm=0$. A submodule $P$ of  an $R$-module $M$ is a \textit{weakly prime submodule} if $M/P$ is a weakly prime $R$-module.
 \end{paragraph}

 \begin{paragraph}\noi
 We  refer to the two aforementioned definitions of prime modules as the global notions of prime modules. The word ``global'' is meant to signify the fact that in these definitions, we use any arbitrary ideals as opposed to fixing an ideal or fixing a given pair of ideals as we do later for their ``local'' counterparts.     An $R$-module $M$ is {\it reduced} \cite{Lee2004} (also called semiprime in \cite{SsevviiriGroenewald2014}) if for all ideals $I$ in $R$ and  elements $m\in M$, $I^2m=0$ implies that $Im=0$. $R$ is a prime ring (i.e., an integral domain) if and only if  $_RR$ is a prime module if and only if $_RR$ is a weakly prime module. $R$ is a reduced ring if and only if $_RR$ is a reduced module.
\end{paragraph}

  \subsection{Locally  prime modules}\label{1.2}

  \begin{paragraph}\noindent
 We now define the local versions of: 1) reduced modules, 2) weakly prime modules and 3) prime modules. The nomenclature is partly motivated by how local cohomology modules are defined, i.e., always defined with respect to a given ideal. These local notions permit the  study of the category theoretic properties of prime modules since each of the ``new primes'' is characterized in terms of functors as we demonstrate later. \end{paragraph}

 \begin{defn}\label{locally}\rm
 Let $I$ and $J$ be fixed ideals of a ring $R$. An $R$-module $M$ is:
  \begin{enumerate}
  	\item $I$-{\it reduced} \cite{SsevviiriI2024} if for all $m\in M$, $I^2m=0$ implies that $Im=0$;
  	\item $I$-\textit{prime} if for all $m \in M$, $Im=0$ implies that either $m=0$ or $IM=0$;
  	\item $(I,J)$-$\textit{prime}$ if for all elements $m\in M$, $IJm=0$
 implies that either $Im=0$ or $Jm=0$.
   \end{enumerate}
\end{defn}

\begin{defn}\label{locally submodules}\rm

A submodule $P$ of an $R$-module $M$ is {\it $I$-semiprime} (resp. {\it $I$-prime} and {\it  $(I,J)$-prime}) if the $R$-module $M/P$ is  $I$-reduced (resp. $I$-prime and $(I,J)$-prime).
	\end{defn}
\begin{paragraph}\noi
    Accordingly, an $R$-module is prime (resp. weakly prime and reduced) if it is $I$-prime (resp. $(I,J)$-prime and $I$-reduced) for all ideals $I$ and $J$ of $R$.
\end{paragraph}
\subsection{Torsion and Completion}
  \begin{paragraph}\noi
Although the  $I$-torsion functor  $\Gamma_I$ and its formal dual the  $I$-adic completion functor $\Lambda_I$ are generally not adjoint on the category of $R$-modules, their derived counterparts $R\Gamma_I$ and $L\Lambda_I$  become adjoint in the derived category for any weakly proregular ideal $I$.  This result, known as the Greenlees-May (GM) Duality \cite{Alonso1997, Porta2014}, constitutes a profound generalization and strengthening of Grothendieck's local duality. Unlike its predecessor, GM Duality holds for any ideal in a Noetherian ring and provides a more conceptual framework; see \cite{Faridian2019} for a comprehensive treatment. An $R$-module $M$ is {\it $I$-torsion} (resp. {\it $I$-adically complete}) if $\Gamma_I(M)=M$ (resp. $\Lambda_I(M)\cong M$).  It is well known that if $R$ is both a Noetherian and an $I$-adically complete ring, then every $R$-module is $I$-adically complete. An $R$-module $M$ for which $IM=0$ is both $I$-torsion and $I$-complete.   $I$-torsion $R$-modules form an abelian category but $I$-adically complete $R$-modules do not form an abelian category in general.  The Matlis-Greenlees-May (MGM) Equivalence asks for when there is  an equivalence of these two subcategories of $R$-modules or of their corresponding derived analogues. Just like the GM-Duality, when the ideal $I$ is weakly proregular,  the MGM  Equivalence was established in the derived category setting  in \cite{Alonso1997} and \cite{Porta2014}.
Recently,  both the GM Duality and the MGM Equivalence were established in the context of $I$-reduced $R$-modules and $I$-coreduced $R$-modules  in  \cite{SsevviiriI2024} and in the derived category setting of $I$-reduced and $I$-coreduced complexes in \cite{Abebaw2025}. In these two cases, the ideal $I$ need not be weakly proregular.
\end{paragraph}

\subsection{Summary of results}
\begin{paragraph}\noi In this paper,  we show that the GM duality  and the MGM equivalence restrict to the subcategories of 
$I$-prime $R$-modules and $I$-coprime $R$-modules, see Theorems \ref{GM} and \ref{MGM} respectively. More properties about these locally prime modules are established. Whereas $I$-prime $R$-modules and $(I,J)$-prime $R$-modules are closed under taking submodules, Proposition \ref{tksubmodules}; their duals $I$-coprime $R$-modules and $(I, J)$-coprime $R$-modules are closed under taking homomorphic images, Proposition \ref{tkimages}.  All locally prime modules and their duals are closed under localization, Proposition \ref{tklocalisations}. We give conditions under which the locally prime modules and their duals are preserved or reversed by taking the Hom functor, Propositions \ref{Hom-reversed1} and \ref{Hom-reversd2}.  In Proposition \ref{Injcogenerator} (resp. Proposition \ref{progenerator}), we give conditions under which $(I,J)$-prime $R$-modules become $(I,J)$-coprime (resp. $(I,J)$-coprime $R$-modules become $(I,J)$-prime  $R$-modules).
\end{paragraph}

\subsection{Rationale for the GM Duality in the context of modules}

\begin{paragraph}\noi 

We have established module subcategories where the $I$-adic completion functor $\Lambda_I$ is the left adjoint of the $I$-torsion functor   $\Gamma_I$. For any $R$-module $M$ which is $I$-coprime  and any $R$-module $N$ which is $I$-prime, the  adjunction   

\begin{equation}\label{ad}
 \text{Hom}_R(\Lambda_I(M), N)\cong \text{Hom}_R(M, \Gamma_I(N))   
\end{equation}

gives a formal categorical bridge between local information of a module around the ideal   $I$  as described in terms of the $I$-torsion functor and the formal completion of a module in its infinitesimal neighborhood of $I$  as described by the $I$-adic completion functor. Furthermore, a morphism $M\rightarrow \Gamma_I(N)$ of $I$-prime $R$-modules defined ``locally" in the support of $I$ corresponds uniquely to a morphism $\Lambda_I(M)\rightarrow N$  of $I$-coprime  $R$-modules defined ``formally" in the neighborhood of  $I$.  This symmetry is particularly noteworthy because, while the $I$-torsion functor $\Gamma_I$ is always left exact, the $I$-adic completion functor $\Lambda_I$ is generally neither left nor right exact. However, on the specified subcategories where the adjunction holds, $\Lambda_I$ becomes right exact. Thus, the adjunction in \eqref{ad} reveals a precise duality between these functors; a symmetry akin to what was previously only visible in the derived category.
\end{paragraph}

\subsection{Structure of the paper}
\begin{paragraph}\noi
The paper consists of five sections. In Section $1$, we introduce $I$-prime modules and $(I,J)$-prime modules. In  Section $2$,  we give properties of $I$-prime modules and $(I,J)$-prime modules which are needed later in the paper. In Section $3$, we introduce the dual notions of  $I$-prime modules and $(I,J)$-prime modules, namely; the  $I$-coprime modules and $(I,J)$-coprime modules respectively and give their properties. In Section $4$,  we give the main results of the paper, i.e.,  Theorems \ref{GM} and \ref{MGM}. In Section  $5$, we give more functorial properties of the $I$-prime modules and $(I,J)$-prime modules as well as their corresponding dual notions. In the same section, we give properties on the globally prime modules which are obtained by the use of locally prime modules, see Corollary \ref{did}. In the conclusion, we give the different studies this paper could lead to.
\end{paragraph}

\section{$I$-prime modules $\&$ $(I, J)$-prime modules}

      \subsection{$I$-prime $R$-modules}
\begin{paragraph}\noi
    Let $I$ be an ideal of a ring $R$. For an $R$-module $M$ and a submodule $N$ of $M$, we write $(N:_M I)$ to denote the submodule $\left\{m\in M~|~Im\subseteq N\right\}$ of $M$.
\end{paragraph}
  \begin{propn}\label{Iprimep1}\rm 

  Let $I$ be a fixed ideal of $R$ and $M$ be an $R$-module. The following statements are equivalent:

  	\begin{enumerate}
  		\item $M$ is $I$-prime;
        \item for every submodule $N$ of $M$, $IN=0$ implies that either $N=0$ or $IM=0$;
  		\item either $(0:_M I)=0$ or $(0:_M I)=M$, i.e., $(0:_M I)$ is a trivial submodule of $M$;
        \item $\text{Hom}_R(R/I,M)\cong 0$ or $\text{Hom}_R(R/I,M)\cong M$.
  		
  		\end{enumerate}

  	  \end{propn}

  \begin{prf}
\noindent
$(1)\Rightarrow(2)$:
Assume that $M$ is $I$-prime and let $N\leq M$ satisfy $IN=0$.
If $N\neq 0$, choose $0\neq m\in N$. Then $Im=0$, and since $M$ is
$I$-prime, we have $m=0 \quad \text{or} \quad IM=0$. Since $m\neq 0$, it follows that $IM=0$. Hence $N=0$ or $IM=0$.     
$(2)\Rightarrow(3)$: Let $N=(0:_M I)$. By definition, $IN=0$. By $(2)$, either $N=0$ or $IM=0$. If $IM=0$, then every element of $M$
is annihilated by $I$, hence $N=M$. Therefore, $ (0:_M I)=0 \quad \text{or} \quad (0:_M I)=M$. 
$(3)\Rightarrow(1)$: Assume that $(0:_M I)=0$ or $(0:_M I)=M$. Let $m\in M$ with $Im=0$. Then $m\in(0:_M I)$.  
If $(0:_M I)=0$, then $m=0$. If $(0:_M I)=M$, then $IM=0$. Hence $M$ is $I$-prime.  $(3)\Leftrightarrow (4)$: Follows from  the fact that $\text{Hom}_R(R/I,M)\cong (0:_M I).$   		
  	\end{prf}

    \begin{cor}\rm  An $R$-module $M$ is a prime $R$-module if and only if for every ideal $I$ of $R$, the submodule $(0:_MI)$ of $M$ is a trivial submodule. 
    \end{cor}

\begin{paragraph}\noi
    The $I$-torsion functor $\Gamma_I$ associates to each $R$-module $M$, a submodule

  \begin{equation}\label{I-torsion1}
\Gamma_I(M):= \left\{ m\in M~ |~   I^km=0 ~\text{for some}~k\in \Z^+\right\}=\bigcup_{k\in \Z^+} \left(0:_M I^k\right).
\end{equation}
\end{paragraph}

  \begin{cor}\label{Iprmc1}\rm
 	If $M$ is an $I$-prime $R$-module and $(0:_M I)\neq 0$, then $M$ is $I$-torsion.
 \end{cor}

 \begin{prf}
 	$(0:_M I)\subseteq\varGamma_I (M)\subseteq M$.
 	If $M$ is $I$-prime and $(0:_M I)\neq 0$, then
 	$(0:_M I)=M$. So, $(0:_M I)=\varGamma_I (M)= M$ and $M$ is $I$-torsion.
 \end{prf}

 	\begin{exam}\rm
 		For any ring $R$, every  $R$-module is trivially  $0$-prime, and vacuously  $R$-prime since  $Rm=M$ for all  $m\in M$.
 	\end{exam}
\begin{paragraph}\noi
For an $R$-module $M$, we write $(0:_R M)$ to denote the ideal of $R$ given by $\left \{r\in R~ |~ rM=0 \right\}$.
\end{paragraph}
 \begin{propn}\label{anhltr prime}\rm Any $R$-module $M$ is $(0:_R M)$-prime.
 \end{propn}

 \begin{prf}
 Since	$(0:_R M)M=0$, $(0:_M (0:_R M))=M$. Then by Proposition \ref{Iprimep1}, $M$ is $(0:_R M)$-prime.
 \end{prf}

 \begin{cor}\rm For any ideal $I$ of a ring $R$ and for any submodule $P$ of an $R$-module $M$, we have:

 \begin{enumerate}
  \item $R/I$ is an $I$-prime $R$-module,
  \item $M/P$ is a $(P:_R M)$-prime $R$-module.
  
   \end{enumerate}
 \end{cor}
\begin{exam}\rm
Every simple module is prime. Since for all ideals $I\subseteq R$, $(0:_M I)$ is a submodule of $M$ and $M$ simple, implies that either $(0:_M I)=0$ or $(0:_M I)=M$. By Example 2.1, every module over a simple ring is also prime.
\end{exam}

\begin{exam}\rm
Every $I$-torsionfree module $M$ (i.e., $M$ is such that $\varGamma_I(M)=0$) is an $I$-prime module. This is because  $(0:_M I)\subseteq \varGamma_I(M)$. So, if $\varGamma_I(M)=0$, then $(0:_M I)=0$. By Proposition \ref{Iprimep1}, $M$ is $I$-prime.
 	\end{exam}
\begin{paragraph}\noi   While studying $I$-reduced $R$-modules in the papers \cite{Kyomuhangi2020, Kyomuhangi2023, SsevviiriI2024,  SsevviiriII2025}, the $I$-torsion functor  $\Gamma_I$ given in (\ref{I-torsion1}) proved very useful. For instance, we have Proposition \ref{Irem}.
\end{paragraph}
\begin{propn}\rm\cite[Proposition 2.2]{SsevviiriI2024}\label{Irem}
  For any $R$-module $M$ and an ideal $I$ of $R$, the following statements are equivalent: 
   \begin{enumerate}
       \item $M$ is $I$-reduced,
       \item $(0:_M I)=(0:_M I^2)$,
       \item $\text{Hom}_R(R/I,M)=\text{Hom}_R(R/I^2,M)$,
       \item $\varGamma_I(M)=\text{Hom}_R(R/I,M)$,
       \item $I\varGamma_I(M)=0$.
   \end{enumerate}
\end{propn}
\begin{paragraph}\noi
    Many of the statements in Proposition \ref{Irem} will be utilised in this paper. We will also show in Section $4$ that the functor $\Gamma_I$ characterises $I$-prime $R$-modules.
\end{paragraph}
 	
 \subsection{ $(I,J)$-prime modules}

\begin{propn}\rm\label{IJPropn1}
    Let $(I,J)$ be a fixed pair of ideals of a ring $R$ and $M$ be an $R$-module. The following statements are equivalent:
	\begin{enumerate}
		\item $M$ is $(I, J)$-prime, i.e., if $IJm=0$, then either $Im=0$ or $Jm=0$;
		
		\item either $(0:_M IJ)=(0:_M I)$ or $(0:_M IJ)=(0:_M J) $; 
		\item  $\text{Hom}_R(R/IJ,M)=\text{Hom}_R(R/I,M)$ or $ \text{Hom}_R(R/IJ,M)=\text{Hom}_R(R/J,M)$. 
		
		\end{enumerate}
\end{propn}
\begin{prf}
    Elementary.
\end{prf}

 \begin{defn}\rm
 An ideal $P$ of a ring $R$ is $I$-\textit{prime} (resp. $(I,J)$-\textit{prime}) if $P$ is an $I$-prime (resp. $(I,J)$-prime) submodule of the $R$-module $R$.
 \end{defn}
 \begin{defn}\rm
 Let $I$ be a fixed ideal of a ring $R$. $P$ is an $I$-\textit{semiprime} ideal of $R$ if whenever $I^2\subseteq P$, then  $I\subseteq P$.
  \end{defn}

    \begin{propn}\rm\label{diagram1} For any fixed pair of  ideals $(I, J)$ of a  ring $R$ and for any submodule $P$ of an $R$-module $M$,  we have Chart  \ref{imp} of implications.

    \end{propn}

\begin{table}[ht]
  \begin{tabular}[h]{cccccc}
&&  ${M/P}$ is  $I$-reduced module & $\Rightarrow$ & $(P:M)\lhd R$ is \\
 & &                              &      &  $I$-semiprime\\

 &	& \rotatebox{90}{$\Rightarrow$}  &          & \rotatebox{90}{$\Rightarrow$} \\

&& $M/P$ is a reduced module & $\Rightarrow$ & $(P:M)\lhd R$ is    \\

& &                              &      &  semiprime\\

&&  \rotatebox{90}{$\Rightarrow$}  & & \rotatebox{90}{$\Rightarrow$} \\

 \text{	$P$ is prime}  & $\Rightarrow$ &   $P$ is weakly prime &	 $\Rightarrow$ &  $(P:M)\lhd R$ is    \\

        &&                       &       &prime\\

 	$\Downarrow $&& $\Downarrow$ && $\Downarrow $ \\

$P$ is $I$-prime & $\Rightarrow$ &  $P$ is $(I,J)$-prime & $\Rightarrow$ & $(P:M)\lhd R$ is   \\

& &                              &      &  $(I,J)$-prime\\

 		 $\Downarrow$      \\
 		 
 	    $M/P$ is $I$-reduced	& \\
 \end{tabular}

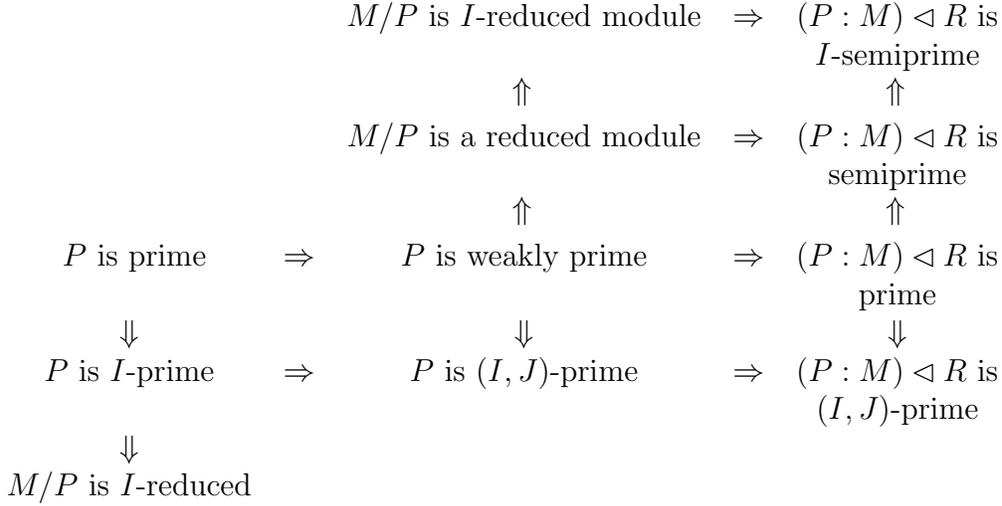
 \caption{Chart of implications between different prime modules.}

\label{imp}
\end{table}

  \begin{prf}
   We only prove that  a submodule $P$ of $M$ is $I$-prime implies $P$ is $(I, J)$-prime and  $P$  is $I$-prime implies that the $R$-module $M/P$ is $I$-reduced. The rest of the implications are either trivial or well known. 
   Let $(I, J)$ be a fixed pair of ideals of a ring $R$. Suppose that $P$ is an $I$-prime submodule of $M$ and $IJm\subseteq P$. By the definition of $I$-prime submodule and Proposition \ref{Iprimep1}, either $Jm\subseteq P$ or $IM\subseteq P$. So, $Jm\subseteq P$ or $Im\subseteq IM \subseteq P$. Suppose that $I^2m\subseteq P$ for some $m\in M$, i.e., $I(Im)\subseteq P$. Then by  Proposition \ref{Iprimep1},  either $Im\subseteq P$ or $IM\subseteq P$. In both cases $Im\subseteq P$. So, $P$ is an $I$-semiprime submodule of $M$ and $M/P$ is an $I$-reduced $R$-module.
	\end{prf}

   \begin{propn}\label{tksubmodules}\rm  Both  $I$-prime $R$-modules and $(I, J)$-prime $R$-modules are closed under taking submodules.
 \end{propn}

 \begin{prf} Let $N$ be a submodule of an $R$-module $M$ and $n\in N$ such that $In=0$. If $M$ is $I$-prime, then either $n=0$ or $IM=0$. Therefore, either $n=0$ or $IN=0$. Now suppose that $M$ is an $(I, J)$-prime module and $IJn=0$ for $n\in N$. Just like in the first case, $n\in M$. Since $M$ is $(I, J)$-prime, either $In=0$ or $Jn=0$ and this shows that $N$ is also an $(I, J)$-prime  $R$-module.
 \end{prf}
 
\begin{paragraph}\noi

 Examples \ref{exIp-P} and \ref{exIr-Ip} demonstrate that the reverse implications in Chart \ref{imp} do not hold in general.
 \end{paragraph}
 \begin{exam}\label{exIp-P}\rm Let $k$ be a field. For ideals $I:=(x^2)$ and $J:=(x)$ of a ring $R:=k[x]$, the $R$-module $M:=k[x]/(x^2)$ is both  $I$-prime and $(I,J)$-prime. However, it is neither prime nor $J$-prime. So, $I$-prime $\not\Rightarrow$ prime and $(I,J)$-prime $\not\Rightarrow$ $J$-prime.  In addition, since $J^2m=0$ for all $m\in M$, but $Jm\neq 0$ for $m=1$,  $(I,J)$-prime $\not\Rightarrow$ weakly prime.
        \end{exam}
     
         \begin{exam}\rm\label{exIr-Ip}
             Let $R:=\mathbb{Z}_6$ and consider the $R$-module $M:=\mathbb{Z}_6$. For an ideal $I:=3\mathbb{Z}_6$  of $R$, the $\mathbb{Z}_6$-module is $I$-reduced since $I$ is  idempotent. However, since  $(0:_M I)=2\mathbb{Z}_6$ it is neither zero nor $M$. By Proposition \ref{Iprimep1}, $M$ is not $I$-prime.
         \end{exam}

\section{The dual notions of locally prime modules}

 \subsection{Globally coprime modules}
\begin{paragraph}\noi
 A module $M$ over a ring $R$ is \textit{coprime}~(also called second) if for all $r\in R$, either $rM=0$ or $rM=M$, see \cite{AnsariToroghy2011, Ceken2013, Yassemi2001}. 
 \end{paragraph}
 \begin{propn}\label{global-cpm}\rm
 For any ring $R$ and an $R$-module $M$, the following statements are equivalent:
 \begin{enumerate}
     \item for all $r\in R$, either $rM=0$ or $rM=M$;
     \item for all ideals $I$ of $R$, either $IM=0$ or $IM=M$.
 \end{enumerate}
 \end{propn}
\begin{prf}\rm
  $(2)\Rightarrow (1)$:  $1)$ is a special case of $2)$. Take $I=(r)$ for any $r\in R$ since $rM=(r)M$. $(1)\Rightarrow (2)$: Suppose that $1)$ holds and $I$ is an arbitray ideal of a ring $R$. If $I\subseteq (0:_R M)$, then $IM=0$. Now, suppose that $I\not\subseteq (0:_R M)$. There exists $r\in I$ such that $rM\neq0$. By $1)$, $rM=M$. However, $rM\subseteq IM \subseteq M$. So, $rM=M$ implies that $IM=M$.
  
\end{prf}
\begin{paragraph}\noi
In this paper, we will mostly utilize statement $2)$ of Proposition \ref{global-cpm} which exhibits a nice formal duality with the earlier defined prime modules. As a  formal dual to weakly prime modules and reduced modules, we define weakly coprime modules and coreduced modules respectively. An $R$-module $M$ over a ring $R$ is \textit{weakly coprime} if for any pair of ideals $I$ and $J$ of $R$, either $IJM=IM$ or $IJM=JM$. A module $M$ over a ring $R$ is \textit{coreduced}  if for all ideals $I$ of $R$, $IM=I^2M$, see \cite{SsevviiriI2024}.   It is worth noting that for an ideal $I$ of a ring $R$ and  an $R$-module $M$, the  condition $IM=0$ in the definition of an $I$-coprime $R$-module $M$ coincides with the condition $(0:_MI)=M$ in the definition of an $I$-prime $R$-module $M$. It follows that if $IM=0$, then the $R$-module $M$ is both $I$-prime and $I$-coprime.    
\end{paragraph}

 \subsection{Locally coprime modules}

\begin{paragraph}\noi
  Dualizing  Definition  \ref{locally}, yields Definition \ref{dualdef}.

\end{paragraph}

 \begin{defn}\rm\label{dualdef} Let $I$ and $J$ be fixed ideals of $R$. An $R$-module $M$ is:
 \begin{enumerate} 
 \item  $I$-\textit{coreduced} \cite{SsevviiriI2024} if $IM = I^2 M$;
 \item $I$-\textit{coprime} if either $IM=0$ or $IM=M$, i.e., if $IM$ is a trivial submodule of $M$;
 \item $(I, J)$-\textit{coprime} if either  $IJM= IM$ or $IJM=JM$.
 \end{enumerate}
 \end{defn}

 \begin{propn}\rm
     Let $M$ be an $R$-module and $I$ be an ideal of $R$. The following statements are equivalent:
     \begin{enumerate}
         \item $M$ is $I$-coprime,
         \item either $M/IM= M$ or $M/IM=0$,
         \item either  $R/I \otimes_R M\cong M$ or $R/I \otimes_R M\cong 0$.
     \end{enumerate}
      \end{propn}
\begin{prf}
 Elementary.

\end{prf}
\begin{exam}\rm 
 Every simple $R$-module is coprime  as is every module over a simple ring. Every idempotent ideal $I$ of a ring $R$ is $I$-coprime as an $R$-module as is every ideal $I$ of $R$ such that $I^2=0$.
 \end{exam}  
 \begin{paragraph}\noi
     The $I$-adic completion functor $\Lambda_I$ associates to each $R$-module $M$, an $R$-module 
    \begin{equation}
	    \Lambda_I(M):=\underset{k}{\underleftarrow{\lim}} (M/I^k M)\cong \underset{k}{\underleftarrow{\lim}}(R/I^k\otimes_R M). 
        \end{equation} 
        It was crucial in characterizing $I$-coreduced modules  in \cite{SsevviiriI2024}   (demonstrated by Proposition \ref{Icorem}) as well as in obtaining the applications of $I$-coreduced $R$-modules in \cite{Kyomuhangi2023, SsevviiriII2025}.   We will use the statements from  Proposition \ref{Icorem} throughout the paper.
 \end{paragraph}
\begin{propn}\rm \cite[ovProposition 2.3]{SsevviiriI2024}\label{Icorem}
    For any $R$-module $M$ and an ideal $I$ of $R$, the following statements are equivalent:
    \begin{enumerate}
        \item $M$ is $I$-coreduced,
\item $IM = I^2M$,
\item $R/I\otimes_R M = R/I^2 \otimes_R M$,
\item $ \Lambda_I(M) = R/I \otimes_R M$,
\item $I\Lambda_I(M)=0$.
    \end{enumerate}

\end{propn}

 \begin{propn}\rm
    Let $(I, J)$  be a fixed pair of ideals of $R$ and let $M$ be an $R$-module. The following statements are equivalent:
     \begin{enumerate}
         \item $M$ is $(I,J)$-coprime, (i.e., either $IJM=IM$ or $IJM=JM$); 
         \item either $M/IJM= M/IM$ or $M/IJM=M/JM$;
         \item either  $R/IJ \otimes_R M\cong R/I \otimes_R M$ or $R/IJ \otimes_R M\cong R/J \otimes_R M$.
         \end{enumerate}
\end{propn}
  \begin{propn}\label{diagram2}\rm Let $(I,J)$ be a fixed pair of ideals of a ring $R$. For any $R$-module $M$, we have the implications in Chart \ref{dualimp} which are dual to the ones in Chart \ref{imp}.
 \end{propn}
 \begin{table}[ht]
 \begin{center}
  \begin{tabular}[h]{cccc}
  
  &	&   \text{${M}$ is $I$-coreduced module}&  \\[6pt]
 & & \rotatebox{90}{$\Rightarrow$}&  \\[6pt]
 	&  & \text{${M}$ is a coreduced module} &  \\[6pt]
 	& & \rotatebox{90}{$\Rightarrow$} &  \\[6pt]
    
 	\text{	$M$ is coprime} &$\Rightarrow$ &\text{ $M$ is weakly coprime} \\[6pt]
    
 	$\Downarrow $&   &  $\Downarrow$ \\[6pt]
    
 	\text{	$M$ is $I$-coprime} 
 	& $\Rightarrow$ &  \text{$M$ is $(I,J)$-coprime} \\[6pt]
    
 	 $\Downarrow$ &&& \\[6pt]
 	
 	\text{$M$ is $I$-coreduced}	 &&&  
 
\end{tabular}
\end{center}
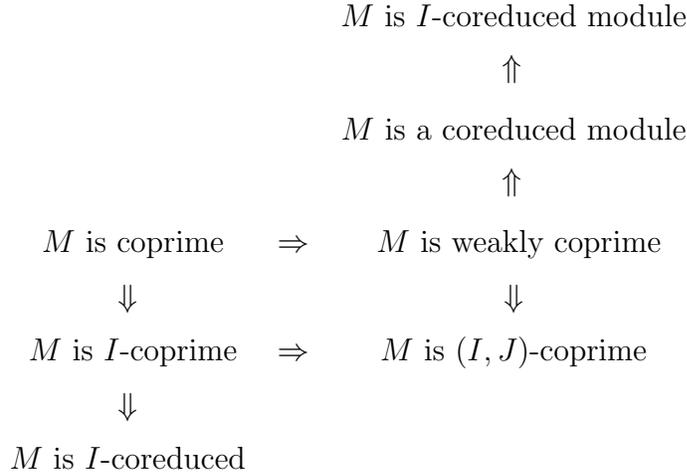
 \caption{Chart of implications between different coprime modules.}

\label{dualimp}
\end{table}
\begin{prf} We only prove the implications: $I$-coprime $\Rightarrow$ $I$-coreduced, and $I$-coprime $\Rightarrow$ $(I,J)$-coprime. The implications: coprime $\Rightarrow$ weakly coprime $\Rightarrow$ coreduced follow from \cite[Corollary 5.4]{AnsariToroghy2011}.  The remaining implications follow from known results. $I$-coprime $\Rightarrow$ $I$-coreduced: if $IM=0$, then $I^2M=I(IM)=0$. Therefore $IM=I^2M$. If $IM=M$, then $I^2M=I(IM)=IM$.   $I$-coprime $\Rightarrow$ $(I,J)$-coprime: If $IM=0$, then for any ideals $I$ and $J$ of $R$, $IJM=J(IM)\subseteq IM=0$ or if $IM=M$ then $IJM=J(IM)=JM$. So, either $IJM=IM$ or $IJM=JM$.
    
\end{prf}
\begin{paragraph}\noi
    The Examples \ref{exICp1},   \ref{exIcp3},  \ref{exIJCp-ICp} and  \ref{exIcor-Icp} demonstrate that the reverse implications in Chart \ref{dualimp} do not hold in general.
\end{paragraph}

     \begin{exam}\label{exICp1}\rm
 Let $R:=\mathbb{Z}$ and $\mathbb{Z}$-module $M:=\mathbb{Z}/6\mathbb{Z}$. If $I:=(6)$ and $J:=(2)$ are ideals of $\mathbb{Z}$, then $M$ is $I$-coprime but it is not coprime. This is because, $IM=0$ but neither $JM=0$ nor $JM=M$.
         \end{exam}

        \begin{exam}\label{exIcp3} \rm
        Let $R:=\mathbb{Z}$ and $M:=\mathbb{Z}/4\mathbb{Z}$. If $I:=(4)$  and $J:=(3)$ are ideals of $R$, then $M$ is $(I,J)$-coprime but it is not weakly coprime. Since for an ideal $K:=(2)$, $K^2M=0$ but $KM\neq 0$.
        \end{exam}
       
        \begin{exam}\label{exIJCp-ICp}\rm
 Consider the ring $R:=\mathbb{Z}$ and the $\mathbb{Z}$-module $M:=\mathbb{Z}$. If   $I:=(2)$ and $J:=\mathbb{Z}$ are ideals of $\mathbb{Z}$, then $M$ is $(I,J)$-coprime but it is not $I$-coprime. Since $IM$ is neither zero nor equal to $M$.  
         \end{exam}
    \begin{exam}\rm\label{exIcor-Icp}
If for an idempotent ideal $I$ of a ring $R$, there exists an $R$-module $M$ such that $IM\neq 0$ and $IM\neq M$, then $M$ is $I$-coreduced but it is not $I$-coprime. Take for instance, the ideal $I:=3\mathbb{Z}_6$ of a ring $R:=\mathbb{Z}_6$  and a $\mathbb{Z}_6$-module  $M:=\mathbb{Z}_6$, $IM=3\mathbb{Z}_6$
which is neither zero nor $M$, but $I^2=I$.    \end{exam}

  \begin{propn}\label{tkimages}\rm  Both  $I$-coprime $R$-modules and $(I, J)$-coprime $R$-modules are closed under taking homomorphic images.
 \end{propn}

 \begin{prf}
  Suppose that $N$  is a submodule of $M$ and $M$ is an $I$-coprime $R$-module. Either $IM=0$ or $IM=M$.  If $IM=0$, then we have $I(M/N)= (IM+N)/N = 0$. Suppose that $IM=M$. Then $I(M/N)= (IM+N)/N = M/N$. So, the factor module $M/N$ is $I$-coprime. Now suppose that $M$ is an  $(I, J)$-coprime $R$-module, i.e., for fixed ideals $I$ and $J$ of $R$, $IJM=IM$ or $IJM=JM$. If $IJM=IM$, then $IJ(M/N)= (IJM+N)/N = (IM+N)/N= I(M/N)$. Similarly, if $IJM=JM$, then $IJ(M/N)= J(M/N)$ and this completes the proof.
 \end{prf}
\section{GM Duality and MGM Equivalence}

  \begin{paragraph}\noi
   In this section, we give the $I$-prime $R$-modules with the  $I$-coprime $R$-modules version of the Greenlees-May Duality and the Matlis-Greenlees-May Equivalence. Let $(R\text{-Mod})_{I\text{-prime}}$, $(R\text{-Mod})_{I\text{-coprime}}$,
    $(R\text{-Mod})_{I\text{-red}}$ and $(R\text{-Mod})_{I\text{-cor}}$ denote the full subcategory of $R$-modules which are $I$-prime, $I$-coprime, $I$-reduced and $I$-coreduced respectively.
  \end{paragraph}

\begin{propn}\label{Ip-Icp}\rm
    If $M$ is an $I$-prime (resp. $I$-coprime) $R$-module, then the submodule of $M,~ \varGamma_I (M)$ (resp. the $R$-module $\Lambda_I(M))$ is an $I$-coprime (resp. an $I$-prime) $R$-module. In particular, the Functors   (\ref{F1}) and (\ref{F2}) below exist.

    \begin{equation}\label{F1}
      \Gamma_I:(R\text{-Mod})_{I\text{-prime}}\longrightarrow(R\text{-Mod})_{I\text{-coprime}}
    \end{equation} 
    \begin{equation}\label{F2}
    \Lambda_I :(R\text{-Mod})_{I\text{-coprime}}\longrightarrow(R\text{-Mod})_{I\text{-prime}}.
     \end{equation}
\end{propn}

\begin{prf}
    If $M$ is $I$-prime, then by Chart \ref{imp}, $M$ is $I$-reduced. By \cite[Proposition 2.2]{SsevviiriI2024}, $I\Gamma_I (M)=0$. It follows by Definition \ref{dualdef} that $\Gamma_I (M)$ is an $I$-coprime $R$-module.    Suppose that $M$ is an $I$-coprime $R$-module. By Chart \ref{dualimp}, $M$ is $I$-coreduced. It follows from \cite[Proposition 2.3]{SsevviiriI2024} that $I\Lambda_I(M)=0$. It follows that, $(0:_{\Lambda_I(M)} I)=\Lambda_I(M)$. By Proposition \ref{Iprimep1}, $\Lambda_I(M)$ is $I$-prime.
\end{prf}

\begin{paragraph}\noi

In \cite[Theorem  3.4]{SsevviiriI2024}, it was proved that the functors \begin{equation}\label{ItortionF}
    \varGamma_I:(R\text{-Mod})_{I\text{-red}}\longrightarrow(R\text{-Mod})_{I\text{-cor}}
\end{equation}
and \begin{equation}\label{IadcF}
    \Lambda_I:(R\text{-Mod})_{I\text{-cor}}\longrightarrow(R\text{-Mod})_{I\text{-red}}
\end{equation} 
exist and form an adjoint pair, with $\varGamma_I$ the right adjoint of $\Lambda_I$. We now have Theorem \ref{GM} which is a restriction of \cite[Theorem 3.4]{SsevviiriI2024} on locally prime modules and locally coprime modules.
\end{paragraph}

  \begin{thm}[GM Duality for locally prime modules]\label{GM}  The  adjointness of the  functors (\ref{ItortionF}) and (\ref{IadcF}) restricts to $I$-prime $R$-modules and $I$-coprime $R$-modules as shown in Chart \ref{GM duality}.

  \begin{table}[h]
 \begin{center}
  \begin{tabular}[h]{ccc}
    
 $(R\text{-Mod})_{I\text{-red}}$
&$\;\mathrel{\substack{\xlongrightarrow{\;\;\;\varGamma_I\;\;\;} \\[0.3em] \xlongleftarrow[\;\;\;\Lambda_I\;\;\;]{}}}\;$ & $(R\text{-Mod})_{I\text{-cor}}$

\\[1.5em]
\rotatebox{90}{$\subseteq$} && \rotatebox{90}{$\subseteq$}
\\[0.4em]
 $(R\text{-Mod})_{I\text{-prime}}$
&$\;\mathrel{\substack{\xlongrightarrow{\;\;\;\varGamma_I\;\;\;} \\[0.3em] \xlongleftarrow[\;\;\;\Lambda_I\;\;\;]{}}}\;$&
$(R\text{-Mod})_{I\text{-coprime}}$

 \end{tabular}
\end{center}
 \caption{Chart of adjoints.}

\label{GM duality}
\end{table}

       \end{thm}
       \begin{prf}   The restrictions of  $\Gamma_I$ and $\Lambda_I$  to these subcategories are well defined, see    
       Proposition \ref{Ip-Icp}.
           If $M$ is $I$-prime, then by Chart \ref{imp} it is $I$-reduced. So, by \cite[Proposition 2.2]{SsevviiriI2024}, $\varGamma_I(M)\cong \text{Hom}(R/I,M)$ a natural isomorphism in $M$. On the other hand, if $M$ is $I$-coprime, then by Chart \ref{dualimp} it is $I$-coreduced and by \cite[Proposition 2.3]{SsevviiriI2024}, $\Lambda_I(M)\cong R/I\otimes M$ is also a natural isomorphism in $M$. However, it is a well known fact that $\text{Hom}(R/I,-)$ is  right adjoint to $R/I\otimes-$. It then follows that, $\varGamma_I$ given in (\ref{F1}) is  right adjoint to $\Lambda_I$ given in (\ref{F2}). 
       \end{prf}

       \begin{exam}\rm  For a fixed ideal $I$ of a ring $R$, choose two $R$-modules $M$ and $N$ such that $IM=IN=0$. Then $M$ and $N$ are both $I$-prime and $I$-coprime and hence both are $I$-reduced and $I$-coreduced. So, $\Gamma_I(M)= (0:_MI)=M$ and $\Lambda_I(N)\cong N/IN= N$. It follows that $M$ and $N$ satisfy the adjunction $\text{Hom}_R(\Lambda_I(N), M)\cong \text{Hom}_R(N, \Gamma_I(M))$. 
           
       \end{exam}

       \begin{paragraph} \noi
       The Matlis-Greenlees-May (MGM) Equivalence asks for when the subcategory of $I$-torsion $R$-modules is equivalent to the subcategory of $I$-complete $R$-modules. In Theorem \ref{MGM}, we give conditions in terms of locally prime modules and locally coprime modules for which this equivalence holds. This generalizes \cite[Theorem 4.3]{SsevviiriI2024}.
      \end{paragraph}

      \begin{thm}\label{MGM}
      For any fixed ideal $I$ of a ring $R$ and any $R$-module $M,$ the following statements are equivalent:\begin{enumerate}
          \item $M$ is $I$-prime and $I$-torsion,
          \item $M$ is $I$-coprime and $I$-adically complete,
          \item $M$ is $I$-reduced and $I$-torsion,
          \item $M$ is $I$-coreduced and $I$-adically complete,
          \item $IM=0$.
      \end{enumerate}
      \end{thm}
      \begin{prf} 
      Since by Chart \ref{imp}, Chart \ref{dualimp}  and  \cite[Proposition 4.2]{SsevviiriII2025}, for any $R$-module $M$, ($I$-prime + $I$-torsion) $\Rightarrow  $ ($I$-reduced + $I$-torsion) $\Leftrightarrow$ $IM=0$, and ($I$-coprime + $I$-adically complete) $\Rightarrow  $ ($I$-coreduced + $I$-adically complete) $\Leftrightarrow$ $IM=0$, it is enough to show that if $IM=0$, then $M$ is both $I$-prime and $I$-torsion and also $M$ is both $I$-coprime and $I$-adically complete.  
       Suppose that $IM=0$, then $\varGamma_I(M)=M$ and $(0:_M I)=M$. So, $M$ is $I$-prime and $I$-torsion. Again, let $IM=0$. By definition, $M$ is $I$-coprime. Furthermore, $I^kM=0$ for all $k\in \mathbb{Z^+}$. So, $\Lambda_I(M)\cong M$ and $M$ is $I$-adically complete. 
  \end{prf}

\section{More functorial properties}

\begin{propn}\rm\label{Hom-reversed1}
If $M$ is an $(I,J)$-coprime $R$-module, then for any $R$-module $N$, $\text{Hom}_R(M,N)$ is an $(I,J)$-prime $R$-module. The converse holds if $N$ is an injective cogenerator. 
\end{propn}
\begin{prf}
   Let $I$ and $J$ be any ideals of $R$ and suppose that $M$ is $(I,J)$-coprime. Utilizing the Hom-tensor adjunction and Propostions \ref{IJPropn1}, we have \\$\text{Hom}_R(R/I,\text{Hom}_R(M,N))\cong \text{Hom}_R(R/I \otimes M,N)\cong \text{Hom}_R(R/IJ \otimes M,N) \cong \\\text{Hom}_R(R/IJ,\text{Hom}_R(M,N)) $. Hence $\text{Hom}_R(M,N)$ is $(I,J)$-prime. For the converse, $\text{Hom}_R(R/I \otimes M,N)\cong \text{Hom}_R(R/I,\text{Hom}_R(M,N))\cong \text{Hom}_R(R/IJ,\text{Hom}_R(M,N))\cong \text{Hom}_R(R/IJ \otimes M,N)$. By the hypothesis, $N$ is an injective cogenerator. So, $\text{Hom}_R(-,N)$ reflects isomorphisms and $R/I\otimes M\cong R/IJ\otimes M$. Therefore, $M$ is $(I,J)$-coprime.
\end{prf}
\begin{propn}\rm\label{Hom-reversd2}
Let $I$ and $J$ be finitely generated ideals  of $R$ and $N$ be an injective $R$-module. If $M$ is  an $(I,J)$-prime $R$-module, then the $R$-module $\text{Hom}_R(M,N)$ is $(I,J)$-coprime.
\end{propn}
\begin{prf} By \cite[Lemma 1.46] {SchenzelSimon2018}, for any finitely generated ideals $I$ and $J$ of $R$,
$R/I\otimes_R \text{Hom}_R(M,N)\cong \text{Hom}_R(\text{Hom}_R(R/I,M),N)\cong \text{Hom}_R(\text{Hom}_R(R/IJ,M),N)$ $\cong R/IJ~ \otimes_R$ $\text{Hom}_R(M,N)$ or $R/J~\otimes_R \text{Hom}_R(M,N)\cong \text{Hom}_R(\text{Hom}_R(R/J,M),N)\cong \text{Hom}_R(\text{Hom}_R(R/IJ,M),N)$ $\cong R/IJ~ \otimes_R$ $\text{Hom}_R(M,N)$. So, $\text{Hom}_R(M,N)$ is $(I,J)$-coprime.
\end{prf}
\begin{propn}\label{tklocalisations}\rm
    If $M$ is an $I$-prime (resp. $(I,J)$-prime, $I$-coprime and $(I,J)$-coprime), then for all prime ideals $\mathcal{P}$ of $R$, the $R_\mathcal{P}$-module $M_\mathcal{P}$ is also $I_\mathcal{P}$-prime (resp. $(I_\mathcal{P},J_\mathcal{P})$-prime, $I_\mathcal{P}$-coprime and $(I_\mathcal{P},J_\mathcal{P})$-coprime).
\end{propn}
\begin{prf}
Suppose that $M$ is an $I$-prime $R$-module. By Proposition \ref{Iprimep1}, either $\text{Hom}_R(R/I, M)=0$ or $\text{Hom}_R(R/I, M)=M$. By \cite[Lemma 4.87]{Rotman2009}, either \\ $\text{Hom}_{R_{\mathcal{P}}}(R_{\mathcal{P}}/I_{\mathcal{P}}, M_{\mathcal{P}})=0$ or $\text{Hom}_{R_{\mathcal{P}}}(R_{\mathcal{P}}/I_{\mathcal{P}}, M_{\mathcal{P}})=M_{\mathcal{P}}$.
Suppose that $M$ is an $(I, J)$-prime  $R$-module. By Proposition \ref{IJPropn1}, for any ideals $I$ and $J$ of $R$,  either $\text{Hom}_R(R/IJ, M)\cong\text{Hom}_R(R/I, M)$ or $\text{Hom}_R(R/IJ, M)\cong \text{Hom}_R(R/J, M)$. By \cite[Lemma $4.87$]{Rotman2009}, either $\text{Hom}_{R_{\mathcal{P}}}(R_{\mathcal{P}}/I_{\mathcal{P}}J_{\mathcal{P}}, M_{\mathcal{P}})\cong \text{Hom}_{R_{\mathcal{P}}}(R_{\mathcal{P}}/I_{\mathcal{P}}, M_{\mathcal{P}})$ or \\$\text{Hom}_{R_{\mathcal{P}}}(R_{\mathcal{P}}/I_{\mathcal{P}}J_\mathcal{P}, M_{\mathcal{P}})=\text{Hom}_{R_{\mathcal{P}}}(R_{\mathcal{P}}/J_{\mathcal{P}}, M_{\mathcal{P}})$. Now, suppose that $M$ is an $I$-coprime $R$-module. Either  $R/I\otimes_R M=0$ or $R/I\otimes_R M=M$. By \cite[Proposition 4.84]{Rotman2009},  either  $R_{\mathcal{P}}/I_{\mathcal{P}}\otimes_{R_{\mathcal{P}}} M_{\mathcal{P}}=0$ or $R_{\mathcal{P}}/I_{\mathcal{P}}\otimes_{R_{\mathcal{P}}} M_{\mathcal{P}}=M_{\mathcal{P}}$. Suppose that $M$ is an $(I,J)$-coprime $R$-module. For any ideals $I$ and $J$ of $R$, either $R/IJ\otimes_R M=R/I\otimes_R M$ or $R/IJ\otimes_R M=R/J\otimes_R M$. By \cite[Proposition 4.84]{Rotman2009},   $R_{\mathcal{P}}/I_\mathcal{P}J_\mathcal{P}~\otimes_{R_\mathcal{P}} M_\mathcal{P}=R_\mathcal{P}/I_\mathcal{P}\otimes_{R_\mathcal{P}} M_\mathcal{P}$
or $R_{\mathcal{P}}/I_\mathcal{P}J_\mathcal{P}~\otimes_{R_\mathcal{P}} M_\mathcal{P}=R_\mathcal{P}/J_\mathcal{P}\otimes_{R_\mathcal{P}} M_\mathcal{P}$.
\end{prf}
\begin{propn}\rm\label{Injcogenerator}
       An injective cogenerator $R$-module which is also $I$-reduced (resp. $(I,J)$-prime) is $I$-coreduced (resp. $(I,J)$-coprime).
   \end{propn} 
    \begin{prf}\rm
    Let $N$ be an injective cogenerator $R$-module which is $I$-reduced. By \cite[Proposition 2.2]{SsevviiriI2024}, $\text{Hom}_ R(R/I,N)\cong \text{Hom}_ R(R/I^2,N)$. By the Hom-tensor adjunction, $\text{Hom}_R(N\otimes_R R/I,N)\cong \text{Hom}_R(N,\text{Hom}_R(R/I,N))\cong \text{Hom}_R(N,\text{Hom}_R(R/I^2,N))\cong \text{Hom}_R(N\otimes_R R/I^2,N) $.  By hypothesis, $N$ is an injective cogenerator. So, $\text{Hom}_ R(-,N)$ reflects isomophisms. Therefore, $R/I \otimes_R N \cong R/I^2 \otimes_R N$ and by  \cite[Proposition 2.3]{SsevviiriI2024}, $N$ is $I$-coreduced. Now suppose that $N$ is an injective cogenerator $R$-module which is $(I,J)$-prime. Let $\text{Hom}_ R(R/IJ,N)\cong \text{Hom}_ R(R/I,N)$. By the Hom-tensor adjunction and for any ideals $I$ and $J$ of $R$ , $\text{Hom}_ R(N\otimes_R R/I,N)\cong \text{Hom}_R(N,\text{Hom}_R(R/I,N))\cong \text{Hom}_R(N,\text{Hom}_R(R/IJ,N))\cong\text{Hom}_ R(N\otimes_R R/IJ,N)$. By hypothesis, $N$ is an injective cogenerator. So, $\text{Hom}_ R(-,N)$ reflects isomorphisms and $R/I \otimes_R N \cong R/IJ~\otimes_R N$. Similarly, if $\text{Hom}_ R(R/IJ,N)\cong \text{Hom}_ R(R/J,N)$, then $R/J \otimes_R N \cong R/IJ \otimes_R N$.
    \end{prf}
    
    \begin{propn}\rm\label{progenerator}
     A progenerator $R$-module which is $I$-coreduced (resp. $(I,J)$-coprime) is also $I$-reduced (resp. $(I,J)$-prime).   
    \end{propn}
    \begin{prf}\rm
    Let $M$ be a  progenerator $R$-module which is $I$-coreduced. Suppose that  $R/I \otimes_R M \cong R/I^2 \otimes_R M$. By the Hom-tensor adjunction we have,$\\\text{Hom}_R(M,\text{Hom}_R(R/I,M))\cong \text{Hom}_ R(M\otimes_R R/I,M)\cong \text{Hom}_ R(M\otimes_R R/I^2,M)\\\cong \text{Hom}_R(M,\text{Hom}_R(R/I^2,M))$. By hypothesis, $M$ is a progenerator. So, \\$\text{Hom}_ R(M,-)$ reflects isomorphisms and $\text{Hom}_ R(R/I,M)\cong \text{Hom}_ R(R/I^2,M)$. Now suppose that $M$ is a  progenerator $R$-module which is $(I,J)$-coprime.
     Let $R/I \otimes_R M \cong R/IJ~ \otimes_R M$. By the Hom-tensor adjunction, for any ideals $I$ and $J$ of $R$, \\$\text{Hom}_R(M,\text{Hom}_R(R/I,M))\cong \text{Hom}_ R(M\otimes_R R/I,M)\cong \text{Hom}_ R(M\otimes_R R/IJ,M)\cong \text{Hom}_R(M,\text{Hom}_R(R/IJ,M)) $. By hypothesis, $M$ is a progenerator. So, $\text{Hom}_ R(M,-)$ reflects isomorphisms. Therefore, $\text{Hom}_ R(R/I,M)\cong \text{Hom}_ R(R/IJ,M)$. Similarly, if $R/J \otimes_R M \cong R/IJ~ \otimes_R M$, then $\text{Hom}_ R(R/J,M)\cong \text{Hom}_ R(R/IJ,M)$.
    \end{prf}

\begin{cor}\rm\label{did}
       Each of the following statements holds:
       \begin{enumerate}
           \item If an $R$-module $M$ is coprime and $N$ is any $R$-module, then the $R$-module $\text{Hom}_R(M,N)$ is prime. The converse holds when $N$ is an injective cogenerator.
           \item Let $R$ be a Noetherian ring and $N$ be an injective cogenerator $R$-module. If $M$ is a prime $R$-module, then $\text{Hom}_R(M,N)$ is a coprime $R$-module.
           \item If $M$ is a prime (resp.  weakly prime, coprime and weakly coprime) $R$-module, then for all prime ideals $\mathcal{P}$ of $R$, the $R_{\mathcal{P}}$-module $M_{\mathcal{P}}$ is prime (resp. weakly prime, coprime and weakly coprime).
           \item An injective cogenerator $R$-module which is reduced (resp. weakly prime) is also coreduced (resp. weakly coprime).
           \item A progenerator $R$-module which is coreduced (resp. weakly coprime) is also reduced (resp. weakly prime).
       \end{enumerate}
    \end{cor}
 \begin{prf}\rm
 This follows from Propositions \ref{Hom-reversed1}, \ref{Hom-reversd2}, \ref{tklocalisations}, \ref{Injcogenerator} and \ref{progenerator} respectively and the fact that modules are globally prime if and only if they are locally prime for each ideal of the ring $R$.
 \end{prf}

 \begin{rem}\rm 
Corollary \ref{did} retrieves some of the results that appear in \cite[Theorem 3.1]{Yassemi2001}.
 \end{rem}    

 \subsection{Conclusion}
 \begin{paragraph}\noi
  Whereas this paper mainly studies locally prime modules, we have also demonstrated that locally prime modules are a pathway to understanding globally prime modules. In addition, it is easy to see that the definitions of the locally prime modules allow for natural extensions of the notions of primeness to other categories. These categories include: the category of chain complexes of $R$-modules \text{Ch}$(R)$, the category of groups, the derived category \text{D}$(R)$ of $R$-modules, and the category of $G$-sets for some group  $G$.  In general, they provide a way to extend and study the notion of primeness to categories that are defined with an action by some  other structure. For the group of $G$, the action is on itself.
     On the other hand there is a disadvantage working with locally prime modules. They are rigid, not as flexible as their global counterparts which permit use of arbitrary ideals. Accordingly, several phenomena that work with global primes  cannot work with locally prime modules. These include the inability to develop a theory of locally prime spectrum of modules as is the case for the global prime modules, see \cite{Lu1999,McCasland1997,Parsa2024,Tekir2005}.
\end{paragraph}
\begin{qn}\rm
    Can one define functors analogous to $\Gamma_I$ and $\Lambda_I$ for a pair of ideals $(I,J)$ of a ring $R$ such that an adjunction exists between the subcategories of $(I,J)$-prime $R$-modules and $(I,J)$-coprime $R$-modules?
\end{qn}

\section*{Acknowledgement}
   \begin{paragraph}\noi
The first author would like to acknowledge the University of Dodoma for the financial  support during her stay at Makerere University. The second author was supported by the International Science Programme through the Eastern Africa Algebra Research Group.  
   \end{paragraph}
     
 \section*{Disclosure of interest} 
 \begin{paragraph}\noi 
   The authors report there are no competing interests to declare.       
 \end{paragraph}

\renewcommand{\bibname}{References}
\nocite{*}
\bibliographystyle{plain}
\bibliography{bib} 
\end{document}